\documentclass[fleqn,preprint,3p,a4paper]{elsarticle}
\usepackage{amssymb}
\usepackage{amsmath}
\usepackage{amsthm}
\usepackage{dcolumn}
\usepackage{endnotes}
\usepackage{tabularx}
\usepackage[matrix,arrow]{xy}
\usepackage{wasysym}

\newtheorem{theorem}{Theorem}[section]
\newtheorem{proposition}[theorem]{Proposition}
\newtheorem{lemma}[theorem]{Lemma}
\newtheorem{corollary}[theorem]{Corollary}

\theoremstyle{definition}
\newtheorem{definition}[theorem]{Definition}

\newtheorem{remark}[theorem]{Remark}

\newcommand{\ir}{{\mathsf{Irr}}}
\newcommand\twoheaduparrow{\mathord{\rotatebox[origin=c]{90}{$\twoheadrightarrow$}}}
\newcommand\twoheaddownarrow{\mathord{\rotatebox[origin=c]{90}{$\twoheadleftarrow$}}}
\newcommand{\dda}{\twoheaddownarrow}
\newcommand{\dua}{\twoheaduparrow}

\newcommand{\cl}{{\rm cl}}
\newcommand{\ii}{{\rm int}}
\newcommand{\SI}{{\rm SI}}
\newcommand{\I}{{\rm I}}

\newcommand{\ua}{\mathord{\uparrow}}
\newcommand{\da}{\mathord{\downarrow}}

\journal{Open Mathematics}

\begin{document}

\begin{frontmatter}

\title{On GSI$_2$-convergence in $T_0$-spaces \tnoteref{t1}}
\author[X. Wen]{Xinpeng Wen}
\address[X. Wen]{School of Mathematics and Information Science, Nanchang Hangkong University, Nanchang 330063, China}
\ead{wenxinpeng2009@163.com}

\author[M. Bao]{Meng Bao}
\ead{mengbao95213@163.com}
\address[M. Bao]{School of Mathematics, Suqian University, Suqian 223800, China}

\author[W. Zhang]{Wenfeng Zhang\corref{mycorrespondingauthor}}
\cortext[mycorrespondingauthor]{Corresponding author}
\address[W. Zhang]{School of Mathematics and Computer Science, Jiangxi Science and Technology Normal University, Nanchang 330038, China}
\ead{zhangwenfeng2100@163.com}

\begin{abstract}

In this paper,we introduce the concept of GSI$_2$-convergence in $T_0$ spaces and the related concept of (strongly) QI$_2$-continuous spaces. It is proved that if GSI$_2$-convergence in $X$ is topological iff $X$ is strongly QI$_2$-continuous for any irreducible complete $T_0$ space $X$.

\end{abstract}


\end{frontmatter}

\section{Introduction}
Convergence and convergence class play an important role in both order theory and general topology (see \cite{GHKLMS, K}). For a topological space $(X, \tau)$ and a class $\mathcal L$ consisting of pairs $((x_i)_{i\in I}, x)$, where $(x_i)_{i\in I}$ is a net in $X$ and $x$ a point of $X$, the topology $\tau$ can naturally induce a convergence class as follows:

\begin{center}
$\mathcal{C}(\tau)=\{((x_i)_{i\in I}, x) : (x_i)_{i\in I} \mbox{ is a net in } X, x\in X \mbox{ and for any } U\in\mathcal \tau$, $x\in U$ implies that $(x_i)_{i\in I} \mbox{ is eventually in } U\}$.
\end{center}

\noindent And we can define a topology on $X$ associated with $\mathcal L$:

\begin{center}
$\mathcal O(\mathcal L)=\{U\subseteq X : ((x_i)_{i\in I}, x)\in\mathcal L$ and $x\in U$ imply
 $x_i\in U$ eventually$\}$.
\end{center}
It is easy to verify that $\tau=\mathcal O(\mathcal C(\tau))$. However, if $\mathcal L$ is not a convergence class in the sense of Kelley \cite{K}, then the convergence class $\mathcal C(\mathcal O(\mathcal L))\neq\mathcal L$, that is, the class $\mathcal L$ is not topological.

Numerous researchers have studied various types of convergences (see [1, 3--8]). With different convergence, they have not only proposed the corresponding continuity of posets (more generally, topological spaces) but also presented some links between order theory and topology. In \cite{GHKLMS}, it was proved that the lim-inf convergence in a dcpo $P$ is topological iff the poset $P$ is a continuous domain. This result was generalized to partially ordered sets (posets) in \cite{ZZ}. In \cite{CK}, it was proved that the GS-convergence in a dcpo $P$ is topological iff the poset $P$ is a quasi-continuous domain. In \cite{RX}, using the cut operator instead of joins, Ruan and Xu introduced and discussed $\mathcal S$-convergence and $\mathcal{GS}$-convergence in posets. They proved that a poset $P$ is s$_2$-continuous (resp., s$_2$-quasicontinuous) iff the $\mathcal S$-convergence (resp., the $\mathcal{GS}$-convergence) in $P$ is topological.

In his invited address at the 2013 Sixth International Symposium on Domain Theory, Jimmie Lawson underscored the imperative of cultivating the core of domain theory directly within $T_0$ spaces, as opposed to confining it to posets. In this direction, by using irreducible sets instead of directed sets, Zhao and Ho \cite{ZH} introduced the SI-topology on $T_0$ spaces as a generalisation of the Scott topology on posets. In \cite{ASHZ}, Andradi et al. defined SI-convergence in $T_0$ spaces and proved that for any $T_0$ space $X$ having condition (I$^*$), $X$ is an I-continuous space iff SI-convergence in $X$ is topological. Subsequently, in \cite{YX} Yang and Xu  introduce the concepts of SI$_2$-convergence and (strongly) I$_2$-continuous space, and proved that SI$_2$-convergence in $X$ is topological iff $X$ is strongly I$_2$-continuous.

In this paper, Similar to the definitions of SI$_2$-convergence and I$_2$-continuous space, we introduce the concept of GSI$_2$-convergence in $T_0$ spaces and the related concept of (strongly) QI$_2$-continuous spaces. We prove that if GSI$_2$-convergence in $X$ is topological iff $X$ is strongly QI$_2$-continuous for $T_0$ space $X$.

\section{Preliminaries}

 In this section, we briefly recall some basic concepts and results about ordered structures and $T_0$ spaces that will be used in the paper. For further details, we refer the reader to [1--2, 10--11].

For a poset $P$ and $A \subseteq P$, define $\ua A=\{x\in P: a\leq x \mbox{ for some }a\in A \}$ and $\da A=\{x \in P: x \leq a \mbox{ for some }a\in A\}$. For $x\in X$, let $\ua x=\ua\{x\}$ and $\da x=\da\{x\}$. A subset $A$ is called a \emph{lower set} (resp., an \emph{upper set}) if $A=\da A$ (resp., $A=\ua A$). If the set of upper bounds of~$A$ has a unique smallest element (that is, the set of upper bounds contains exactly one of its lower bounds), we call this element the \emph{least upper bound} and write it as~$\vee A$ or sup~$A$ (for supremum). Similarly the greatest lower bound is written as~$\wedge A$ or inf~$A$ (for infimum).
Define $A^{\ua}=\{u\in P: A\subseteq \da u\}$ (the sets of all upper bounds of $A$ in $P$) and $A^{\da}=\{v\in P: A\subseteq \ua v\}$ (the sets of all lower bounds of $A$ in $P$). The set $A^{\delta}=(A^{\ua})^{\da}$ is called the $cut$ of $A$ in $P$.
A nonempty subset $D$ of a poset $P$ is called \emph{directed} if every finite subset of $D$ has an upper bound in $D$. The set of all directed sets of $P$ is denoted by $\mathcal{D}(P)$. The poset $P$ is called a \emph{directed complete poset}, or \emph{dcpo} for short, if for any
$D\in \mathcal D(P)$, $\bigvee D$ exists in $P$.
 A subset $U$ of $P$ is \emph{Weakly Scott open} if (i) $U=\mathord{\uparrow}U$, and (ii) for any directed subset $D$ for which $D^{\delta}\cap U\neq\emptyset$  implies $D\cap U\neq\emptyset$. The topology formed by all the Scott open sets of $P$ is called the \emph{Weakly Scott topology}, written as $\sigma_2(P)$.  The upper sets of $P$ form the (\emph{upper}) \emph{Alexandroff topology} $\alpha(P)$. We say that a nonempty family $\mathcal{G}$ of sets is directed if given $G_1$, $G_2$ in the family $\mathcal{G}$, there exists $G\in \mathcal{G}$ such that $G_1$, $G_2\leq G$, i.e., $G\subseteq \uparrow G_2\cap \uparrow G_1$.
  For nonempty subsets $F$ and $G$ of $P$, we say $F$ approximates $G$ if for every directed subset $D\subseteq P$, $D^{\delta}\cap \uparrow G\neq\emptyset$ implies $d\in \uparrow F$ for some $d\in D$.
  A poset $P$ is called a $s_2$-quasicontinuous poset if for all $x\in P$, $\uparrow x$ is the directed (with respect to reverse inclusion) intersection of sets of the form $\uparrow F$, where $F$ approximates $\{x\}$ and $F$ is a finite set of $P$. Let $P^{(<\omega)}$ be the set of all nonempty finite subsets of P.

Let $X$ be a $T_0$ space. Let $\mathcal{O}(X)$ (resp.,$\Gamma(X)$) be the set of all open subsets (resp., closed subsets) of $X$. For a subset of $X$,  denote the closure of $A$ in $X$ by $\cl_{X}~A$ or simply by $\cl~A$ and the interior of $A$ in $X$ by $\ii_{X} A$ in $X$ or simply by $\ii ~ A$. The subspace of $X$ on $A$ is denote by $(A, \mathcal{O}(X)\mid_A)$. A nonempty subset $A$ of a $T_0$ space $X$ is called an \emph{irreducible} set if for any $F_1, F_2\in\Gamma(X)$, $A\subseteq F_1\cup F_2$ implies $A\subseteq F_1$ or $A\subseteq F_2$. We denote by $\ir(X)$  the set of all irreducible  subsets of $X$. The space $X$ is is said to be \emph{irreducible}-\emph{complete} if every irreducible subset of $X$ has a sup. Defined the specialization order of $X$: $x\leq y$ iff $x\in \cl \{y\} $. In the following, when a $T_0$ space is considered as a poset, the order always refers to the specialization order if no other explanation is given.
Let $\mathbf{up}(X)=\{\uparrow_X U : U\subseteq X\}$ (that is, the family of all upper subsets of $X$), $X^{(<\omega)}$ be the set of all nonempty finite subsets of $X$ and $\mathbf{Fin}~\!\!X=\{\ua F : F\in X^{(<\omega)}\}$. For a subset $M$ of $X$, the \emph{induced topology} or \emph{subspace topology} on $M$ is denoted by $\mathcal O(X)|_M$, that is, $\mathcal O(X)|_M=\{U\cap M : U\in \mathcal O(X)\}$. A nonempty family $\mathcal{G}$ of subsets of a set $X$ is called \emph{filtered} if given $G_1$, $G_2$ in the family $\mathcal{G}$, there exists $G_3\in \mathcal{G}$ such that $G_3\subseteq \uparrow G_2\cap \uparrow G_1$.

The following two lemmas are well-known and can be easily verified.

\begin{lemma}\label{subspace-irr-1}
Let $X$ be a space and $Y$ a subspace of $X$. Then the following conditions are equivalent for a
subset $A\subseteq Y$:
\begin{enumerate}[\rm (1)]
	\item $A$ is an irreducible subset of $Y$.
	\item $A$ is an irreducible subset of $X$.
	\item ${\rm cl}_X A$ is an irreducible closed subset of $X$.
\end{enumerate}
\end{lemma}

\begin{lemma}\label{continuous-mapping-irr}
If $f: X\rightarrow Y$ is continuous and $A\in\ir(X)$, then $f(A)\in\ir(Y)$.
\end{lemma}

A subset $W$ of a $T_0$-space $X$ is called \emph{saturated} if $W$ equals the intersection of all open sets containing it (equivalently, $W\in \mathbf{up}(X)$). We shall use $Q(X)$ to denote the set of all nonempty compact saturated subsets of $X$ and endow it with the \emph{Smyth order} $\sqsubseteq$, that is, for $K_1,K_2\in Q(X)$, $K_1\sqsubseteq K_2$ if{}f $K_2\subseteq K_1$. The upper Vietoris topology on $Q(X)$ is the topology that has $\{\Box U: U\in \mathcal{O}(X)\}$ as a base, where $\Box U=\{Q\in Q(X):Q\subseteq U\}$, and the resulting space is called the \emph{Smyth power space} or \emph{upper space} of $X$ and is denoted by $P_S(X)$. For $A\in \Gamma(X)$, define $\Diamond A=\{Q\in Q(X):Q\cap A\neq\emptyset\}$. Clearly, $Q(X)\setminus \Box U=Q(X)\setminus \Diamond (X\setminus U$ for any $U\in\mathcal O(X)$.

\begin{lemma}\label{xi-continuous}
	For a $T_0$ space $X$, then the canonical mapping $\xi_X : X\longrightarrow P_S(X), x\mapsto\uparrow x$, is a dense topological embedding.
\end{lemma}
\begin{proof}
For $U\in\mathcal O(X)$, we have
$$\xi_X^{-1}(\Box U)=\{x\in X : ~\!\!\!\uparrow x\in\Box U\}=\{x\in X : ~\!\!\!\uparrow x\subseteq U\}=U, \hbox{~and}$$.
$$\hbox{~if~} U\neq\emptyset, \hbox{~then~} \emptyset\neq\{\uparrow x : x\in U\}=\xi_X (X)\cap\Box U.$$
\noindent So $\xi_X$ is continuous and $\xi_X (X)$ is dense in $P_S(X)$.
In addition, we have that $\xi_X(U)=\{\uparrow x : x\in U\}
=\{\uparrow x : ~\!\!\!\uparrow x\in\Box U\}
=\xi_X(X)\cap \Box U$,
\noindent which implies that $\xi_X$ is an open mapping to $\xi_X(X)$, as a subspace of $P_S(X))$.
As $X$ is $T_0$, $\xi_X$ is an injection. Thus $\xi_X$ is a dense topological embedding.
\end{proof}

A \emph{net} $(x_i)_{i\in I}$ in a set $X$ is a mapping from a directed set $I$ to $X$.  If $H(x)$ is a property of the elements $x\in X$, we say that $H(x)$ \emph{holds eventually} in the net $(x_i)_{i\in I}$ if there is a $i_0\in J$ such that $H(x_i)$ is true whenever $i_0\leq i$.

\begin{lemma}\emph{(\cite{HK})}\label{T Rudin}
(Topological Rudin Lemma) Let $X$ be a topological space and $\mathcal{A}$ an
irreducible subset of $Q(X)$ which is the the upper Vietoris space of all nonempty compact subsets. Any closed set $C\subseteq X$ that
meets all members of $\mathcal{A}$ contains an irreducible closed subset $A$ that still meets all
members of $\mathcal{A}$.
\end{lemma}

\begin{lemma}\emph{(\cite{WHX})}\label{Rudin irr}
Let $X$ be a $T_0$-space and $\mathcal{F}\subseteq \mathbf{Fin}~\!\! X$. If $\mathcal F\in \ir(P_S(X))$,
 then there exists $A\in \ir(X) $ with $A\subseteq \bigcup \mathcal{F}  $ such that  $A\cap \ua F\neq \emptyset$ for any $\ua F\in \mathcal{F} $.
\end{lemma}

\begin{lemma}\emph{(\cite{WHX})}\label{Rudin irr closed set}
 Let $X$ be a $T_0$-space $X$ and $\mathcal{K}\in \ir(P_S(X))$. If $B$ is a closed set that intersects every element of $\mathcal{K}$. Then $\{\ua (K\cap B) : K\in \mathcal{K}\}\in \ir (P_S(X))$.
\end{lemma}

\begin{lemma}\emph{(\cite{WHX})}\label{Rudin irr open set}
 Let $X$ be a $T_0$-space $X$ and $\mathcal{K}\in \ir(P_S(X))$. If $U$  is an open set satisfying $K\setminus U\neq \emptyset$ for any  $K\in \mathcal{K}$.
 Then $\{\ua (K\setminus U) : K\in \mathcal{K}\}\in\ir (P_S(X))$.
\end{lemma}

\begin{definition} (\cite{SAZ})\label{2dy7}
Let $X$ be a $T_0$-space. A subset $U$ of $X$ is called \emph{SI$_2$-open} if the following two conditions are satisfied:

(1) $U$ is an open set in $X$, and

(2) for any $F\in \ir(X)$, $F^{\delta}\cap U\neq\emptyset$ implies $F\cap U\neq\emptyset$.

The set of all SI$_2$-open sets in $X$ is denoted by $\mathcal O_{\SI_2}(X)$. It is straightforward to verify that $\mathcal O_{\SI_2}(X)$ is a topology on $X$, called the \emph{SI$_2$-topology}. The space $(X, \mathcal O_{\SI_2}(X))$ will also be simply written as SI$_2(X)$.
\end{definition}

\begin{proposition}\label{rudin yunyong}
Let $X$ be a $T_0$ space, $U$ be a SI$_2$-open set, $x\in X$, and $\mathcal{F}=\{\ua F:F\in X^{(<\omega)}\}\in \ir(P_S(X))$. If $\bigcap\mathcal{F}
\subseteq \uparrow x\subseteq U$,then  $F\subseteq U$ for some $\ua F\in\mathcal{F}$.
\end{proposition}

\begin{proof}
Assume, on the contrary, that $\ua F\nsubseteq U$ for all $\ua F\in\mathcal{F}$. Then by Lemma \ref{Rudin irr open set} $\{\ua (F\setminus U) : \ua F\in \mathcal{F}\}\in\ir (P_S(X))$. For any $\ua H\in \mathbf{Fin}~\!\!X$ ($H$ is finite), it is easy to see that $\ua H=\ua min(H)$. Let $G=\bigcup \{min(F\setminus U) : \ua F\in\mathcal F\}$ and $C=\cl~ G$. Then $C\cap \ua (F\setminus U)\neq\emptyset$ for all $\ua F\in\mathcal F$. By Lemma  \ref{T Rudin}, $C$ contains a minimal irreducible closed subset $A$ that still meets all
members of $\{\ua (F\setminus U) : \ua F\in \mathcal{F}\}$. Then $\mathop{\bigcap}\limits_{d \in A}\uparrow d\subseteq  \bigcap \{\ua (F\setminus U) : \ua F\in\mathcal F\}\subseteq \bigcap \mathcal F\subseteq \uparrow x\subseteq U$ and $A^{\delta}\cap U\neq\emptyset$. As $U$ is an SI$_2$-open set, there exists $a\in A$ such that $a\in U$ and hence $a\in U\cap\cl~ G$. So $U\cap G\neq\emptyset$. Select an $u\in U\cap G$. Then $u\in \min(F\setminus U)\subseteq F\setminus U$ for some $\ua F\in\mathcal F$, which contradicts $u\in U$. Therefore, $\ua F\subseteq U$ for some $\ua F\in\mathcal{F}$.

\end{proof}

\begin{definition} (\cite[Definition 3.1]{ZH})\label{SI-open} 
Let $X$ be a $T_0$-space. A subset $U$ of $X$ is called \emph{SI}-\emph{open} if the following two conditions are satisfied:
\begin{enumerate}[\rm (1)]
\item $U$ is an open set in $X$, and

\item for any $F\in \ir(X)$, $\bigvee F\in U$ implies $F\cap U\neq\emptyset$.
\end{enumerate}
The set of all SI-open sets in $X$ is denoted by $\mathcal O_{\SI}(X)$. It is straightforward to verify that $\mathcal O_{\SI}(X)$ is a topology on $X$, called the \emph{SI}-\emph{topology}. The space $(X, \mathcal O_{\SI}(X))$ will also be simply written as $SI(X)$.
\end{definition}

If the topological space is irreducible-complete, we have the following result.

\begin{proposition}\emph{(\cite{WHX})}
Let $X$ be an irreducible-complete $T_0$-space, $U$ be an SI-open set and $\mathcal{F}\subseteq \mathbf{Fin}~\!\! X$. If $\mathcal F\in \ir(P_S(X))$ and $\mathop{\bigcap}\mathcal{F} \subseteq U$, then  $\ua F\subseteq U$ for some $\ua F\in\mathcal{F}$.
\end{proposition}

\section{GSI$_2$-convergence}

In this section, we introduce and study the GSI$_2$-convergence in $T_0$-spaces. We show that for an  $T_0$-space $X$, the GSI$_2$-convergence in $X$ induces the SI$_2$-topology on $X$.

\begin{definition}(\cite{RX}) \label{GELB}
Let $P$ be a poset and $(x_j)_{j\in J}$ a net in $P$. $F\subseteq X$ is called a Generalized eventual lower bound of a net ($(x_j)_{j\in J}$ in $P$, if $F$ is finite and there exists $k\in J$ such that $x_j\in \uparrow F$ eventually.
\end{definition}

\begin{definition}(\cite{RX}) \label{GS-converges}
Let $P$ be a poset, a net $(x_i)_{i\in I}$
in $P$ is GS$_2$-converges to $x\in P$ if and only if there exists a directed
set $\mathcal{F}=\{F: F\in X^{(<\omega)}\}$ such that

(1) $ \mathop{\bigcap}\limits_{F \in \mathcal{F}}\uparrow F\subseteq \uparrow x$ ($a\in X$), and

(2) $\mathcal{F}$ is a set of Generalized eventual lower bounds of $(x_i)_{i\in I}$.
\end{definition}

It is known that the GS$_2$-converges structure induces the weakly Scott topology (see \cite{RX}). The following proposition shows that every set of Generalized eventual lower bounds of a net can be completely described by upper sets.

\begin{proposition}(\cite{WHX})\label{4.3}
Let $P$ be a poset, $\mathcal{G}\subseteq P^{(<\omega)}$ and $(x_i)_{i\in I}$ be a net in $P$. Then the following conditions are equivalent:
\begin{enumerate}[\rm (1)]

\item $\mathcal{G}$ is a family of Generalized eventual lower bounds of $(x_i)_{i\in I}$.

\item For any upper set $U$, $\{\ua G : G\in \mathcal{G}\}\cap \Box U\neq\emptyset$  implies $x_i\in U$ eventually, where $\Box U=\{K\in Q(X) : K\subseteq U\}$.
\end{enumerate}

\end{proposition}

\begin{proof}
(1) $\Rightarrow$ (2): As $\{\ua G : G\in \mathcal{G}\}\cap \Box U\neq\emptyset$, there is $G\in \mathcal G$ such that $\ua G\subseteq U$. By assumption, $x_i\in \uparrow G$ eventually. So $x_i\in U$ eventually.

(2) $\Rightarrow$ (1): For $F\in \mathcal G$, let $U=\ua F$. Then $U$ is an upper set and  $\ua F\in \{\ua G : G\in \mathcal{G}\}\cap \Box U$. By hypothesis, $x_i\in U$ eventually, that is, $x_i\in\uparrow F$  eventually. Thus $F$ is a generalized eventual lower bound of $(x_i)_{i\in I}$.
\end{proof}

By considering Proposition \ref{4.3} and a poset $P$ endowed with the Alexandrov topology on it, the definition of GS$_2$-converges can be rephrased in a topological way as follows: a net $(x_i)_{i \in I}$ in $P$ GS$_2$-converges to $ x \in P $ if and only if
there exists an $\mathcal{F}\subseteq P^{(<\omega)}$ such that $\{\ua F : F\in\mathcal F\}\in \ir (P_S(P,\alpha(P)))$ and the following two conditions are satisfied:

(i) $ \mathop{\bigcap}\limits_{F \in \mathcal{F}}\uparrow F\subseteq \uparrow x$, and

(ii) for any $U\in\alpha(P)$, $\mathcal{F}\cap \square U\neq\emptyset$ implies $x_i\in U$ eventually.

Lifting the above to the realm of \( T_0 \) spaces, we have the following definition.

\begin{definition} \label{GSI-converges}

A net $(x_i)_{i\in I}$ is called \emph{$GSI_2$}-\emph{convergent} to a point $x$ in a~$T_0$-space $X$ if there exists an $\mathcal{F}\subseteq X^{(<\omega)}$ such that $\{\ua F : F\in\mathcal F\}\in \ir (P_S(X))$ and the following two conditions are satisfied:
\begin{enumerate}[\rm (i)]
\item for any $U\in\mathcal O(X)$, $\{\ua F : F\in \mathcal{F}\}\bigcap \square U\neq\emptyset$ implies $x_i\in U$ eventually, and

\item $ \mathop{\bigcap}\limits_{F \in \mathcal{F}}\uparrow F\subseteq \uparrow x$.

\end{enumerate}

\noindent And in this case, we write $(x_i)_{i\in I}\stackrel{GSI_2}\longrightarrow x$. Let $\mathcal{GSI}_2(X)=\{((x_i)_{i\in I}, x) : (x_i)_{i\in I} \mbox{~is~a~net~in~} X, x\in X \mbox{~and~} (x_i)_{i\in I}\stackrel{GSI_2}\longrightarrow x\}$.
\end{definition}

\begin{remark}\label{GSIremark} For a $T_0$ space $X$, we have the following statements:

(1) The constant net $(x)_{i\in I}$ in $X$ with value $x$ GSI$_2$-converges to $x$.

(2) If $(x_i)_{i\in I}\stackrel{GSI_2}\longrightarrow x$ in $X$, then $(x_i)_{i\in I}\stackrel{GSI_2}\longrightarrow y$ for any $y\leq x$. Thus the GSI-convergence points of a net are generally not unique.

(3) Let $P$ be a poset. Then the GSI$_2$-convergence in $(P, \alpha(P))$ coincides with the
GS$_2$-converges in $P$.
\end{remark}

\begin{lemma}\label{Irr net}
Let $X$ be a $T_0$ space. For every $F \in \ir(X)$ and $x\in F^{\delta}$, there exists a net $(x)_{i\in I}\subseteq F$
such that $(x)_{i\in I}$  GSI$_2$-convergence to $x$ .
\end{lemma}

\begin{proof}
Let $x\in F^{\delta}$. Define $I=\{(U,e)\in \mathcal O(X)\times F: e\in U\}$ and equip $I$ with $\leq$ defined as follows:
$(U_1,e_1)\leq (U_2,e_2)$ iff $U_2\subseteq U_1$. Irreducibility of $F$ gives that $I$ is directed. For every $(U,e)\in I$ ,we let $x_{(U,e)} = e$, $\mathcal{F}=\{\{ e\}: e\in F\}$ and $\{\ua F: F\in\mathcal{F}\}\in \ir(P_S(X))$. Therefore, $\cap\{\uparrow e: e\in F\}\subseteq \uparrow x$. If $U\in \mathcal O(X)$ and $\Box U \cap \mathcal{F}\neq\emptyset $,  then there is  $d \in F$ such that $(U,d)\in I$. For every $(V,e)\geq (U,d)$, we have $e\in V\subseteq U$. Hence $(x_i)_{i\in I}\stackrel{GSI_2}\longrightarrow x$.
\end{proof}

The following conclusion tells us that the topology induced by the GSI$_2$-convergence structure is precisely the SI$_2$-topology. This fact can be regarded as a topological parallel of the well-known fact in domain theory:  GS$_2$-convergence structure induces the weakly Scott topology.

\begin{theorem} \label{SItopo=GSItopo}
For any $T_0$ space $X$, the two topologies $\mathcal O(\mathcal{GSI}_{2}(X))$ and $\mathcal O_{\SI_2}(X)$ coincide, that is, $\mathcal O_{\SI_2}(X)=\{U\subseteq P:$ whenever $(x_i)_{i\in I}\stackrel{GSI_2}\longrightarrow x$ and $x\in U$, then  $x_i\in U$  eventuall$\}$ .
\end{theorem}

\begin{proof}

Let $V\in \mathcal O(\mathcal{GSI}_{2}(X))$. Firstly, we show that $V\in \mathcal O (X)$. In fact, suppose not. Then there is $x\in V$ such that $W \nsubseteqq V$ for every $W\in \mathcal N (x)=\{U\in\mathcal O (X):x\in U\}$. We know that $\mathcal N (x)$ which equip with reverse inclusion order is a directed posets. For any $W\in \mathcal N (x)$ let $x_W\in W\setminus V$ to form a net $(x_W)_{W\in \mathcal N (x)}$.
Let $\mathcal F=\{\{x\}\}$. Then $\{\ua F : F\in\mathcal F\}=\{\ua x\}\in \ir(P_S(X))$ and $\cap\{\ua F :F\in\mathcal F\}=\uparrow x$.
Let $U\in \mathcal O (X)$ with $x\in U$. It is easy to  know that   $x_W\in U$ eventually. Hence $(x_W)_{W\in \mathcal N (x)}\stackrel{GSI_2}\longrightarrow x$. As $V\in \mathcal O(\mathcal{GSI}_2(X))$, $x_W\in V$ for some $W\in \mathcal N (x)$, which is a contradiction. Hence $V\in \mathcal O(X)$. Next, we proof that $V\in \mathcal O_{\SI_2}(X)$. Suppose that $F \in Irr(X)$ such that  $x\in F^{\delta}\cap V\neq\emptyset$. By Lemma \ref{Irr net}, there exists a net $(x)_{i\in I}\subseteq F$ such that it GSI$_2$-converges to $x$. As $x\in V$ and $V\in \mathcal O(\mathcal{GSI}_2(X))$, we have $F\cap V\neq \emptyset$.

Now, let  $V\in \mathcal O_{\SI_2}(X)$ and $(x_i)_{i\in I}\stackrel{GSI_2}\longrightarrow x\in V$.
Then $V\in \mathcal O(X)$ and there exists $\mathcal{F}\subseteq X^{(<\omega)}$ such that $\{\ua F : F\in\mathcal F\}\in \ir(P_S(X))$ and conditions (i) and (ii) in Definition \ref{GSI-converges} are satisfied.
By Proposition \ref{rudin yunyong}, there is a $F\in \mathcal{F}$ such that $F\subseteq V$. Hence $\mathcal{F}\cap \Box V\neq \emptyset$ and $x_i\in V$ eventually. So  $V\in \mathcal O(\mathcal{GSI}_2(X))$.
\end{proof}

\begin{lemma}(\cite{SAZ})\label{SI continuous function}
 A continuous mapping $f : (X,\mathcal{ O }(X))\rightarrow (Y,\mathcal{ O }(Y))$ is a continuous mapping between $(X,\mathcal{ O }_{\SI_2}(X))$ and $(Y,\mathcal{ O }_{\SI_2}(Y))$ if and only if $f(F^{\delta})\subseteq (f(F))^{\delta}$ holds for any $F \in Irr(X)$ .
\end{lemma}

\begin{proposition}    \label{3mt5}
Let $X, Y$ be a $T_0$ spaces and $f$ be a continuous mapping from $X$ to $Y$. Then the following two conditions are equivalent:

(1) $f$ is a continuous mapping from $\SI_2(X)$ to $\SI_2(Y)$.

(2) For any net $(x_i)_{i\in I}$ and $x\in X$ , $(x_i)_{i\in I}\stackrel{GSI_2}\longrightarrow x$ in $X$ implies $f(x_i)_{i\in I}\stackrel{GSI_2}\longrightarrow f(x)$ in $Y$.
\end{proposition}

\begin{proof}

(1) $\Rightarrow$ (2): Because $f$ is a continuous mapping, $f$ is order-preserving. Suppose that $(x_i)_{i\in I}\stackrel{GSI}\longrightarrow x$ in $X$. Now we prove that $f(x_i)_{i\in I}\stackrel{GSI_2}\longrightarrow f(x)$ in $Y$. As $(x_i)_{i\in I}\stackrel{GSI_2}\longrightarrow x$,
 there exists $\mathcal{F}\subseteq X^{(<\omega)}$ such that $\{\ua F : F\in\mathcal F\}\in \ir(P_S(X))$ and conditions (i) and (ii) in Definition \ref{GSI-converges} are satisfied.
 Because $\hat{f}^{-1}(\Box U)=\{Q\in P_S(X):\uparrow f(Q)\subseteq U\}=\{Q\in P_S(X):Q\subseteq f^{-1}(U)\}=\Box f^{-1}(U)$, then $\hat{f}:P_s(X)\longrightarrow P_s(Y)(Q\longmapsto \uparrow f(Q))$ is a continuous mapping.
Then $\mathcal{F}^{\backprime}=\{\ua f(F):F\in \mathcal{F}\}\in \ir(P_S(Y))$.
Next, we show that  $\mathop{\bigcap}\limits_{F \in \mathcal{F}}\uparrow f(F)\subseteq \uparrow x$.
 Let $a\in \mathop{\bigcap}\limits_{F \in \mathcal{F}}\uparrow f(F)$.
 Then $\downarrow a \cap f(F)\neq \emptyset$ for any $F\in \mathcal{F}$.
 Hence $f^{-1}(\downarrow a) \cap F \neq \emptyset$ for any $F\in \mathcal{F}$. As $f^{-1}(\downarrow a)$ is closed in $SI_2(X)$, by Lemma \ref{Rudin irr closed set} and Lemma \ref{Rudin irr}  there exists $D\in Irr(X) $ with $D\subseteq \mathop{\bigcup}\limits_{F \in \mathcal{F}} (f^{-1}(\downarrow a) \cap F) $ such that  $D\cap  (f^{-1}(\downarrow a) \cap F)\neq \emptyset$ for any $F\in \mathcal{F} $.
 Hence $f(D)\subseteq \downarrow a$ and $\mathop{\bigcap}\limits_{d \in D} \uparrow  d\subseteq \mathop{\bigcap}\limits_{F \in \mathcal{F}} \uparrow (f^{-1}(\downarrow a) \cap F)$ $\subseteq\mathop{\bigcap}\limits_{F \in \mathcal{F}} \uparrow  F\subseteq \uparrow x$. Thus, $x\in D^{\delta}$ and $f(D)\subseteq \downarrow a$. As $f(D^{\delta})\subseteq (f(D))^{\delta}$, we have $f(x)\in f(D^{\delta})\subseteq(f(D))^{\delta}\subseteq \downarrow a$.
 Therefore, $f(x_i)_{i\in I}\stackrel{GSI_2}\longrightarrow f(x)$ satisfy the conditions (ii) of Definition \ref{GSI-converges}.
 For $\forall U\in \mathcal{O}(Y)$, if $\mathcal{F} ^{\backprime}\cap \Box U\neq\emptyset$, then $\mathcal{F} \cap \Box f^{-1}(U)\neq\emptyset$. As $(x_i)_{i\in I}\stackrel{GSI_2}\longrightarrow x$, we have that $x_i\in f^{-1}(U)$ eventually. hence $f(x_i)\in U$ eventually. So $f(x_i)_{i\in I}\stackrel{GSI_2}\longrightarrow f(x)$.

(2) $\Rightarrow$ (1): Let $V\in\mathcal O_{\SI_2}(Y)$. By the continuity of $f : X \rightarrow Y$, we have $f^{-1}(V)\in\mathcal O(X)$. Let $F\in Irr(X)$ with $x\in F^{\delta}\cap f^{-1}(V)\neq\emptyset$. By Lemma \ref{Irr net}, there is a net $(x)_{i\in I}\subseteq F$
such that it GSI$_2$-convergence to $x$. By the assumption, the net $(f(x_i))_{i\in I}$ GSI$_2$-converges to $f(x)$. By Theorem \ref{SItopo=GSItopo}, $(f(x_i))_{i\in I}$ converges to $f(x)\in V$ with respect to $\mathcal{O}_{\SI_2}(Y)$. This implies $f(x_i)\in V$ eventually. Thus, $x_i\in f^{-1}(V)$ eventually. Hence $f^{-1}(V)\cap F\neq\emptyset$. We conclude that $f^{-1}(V)\in \mathcal O_{\SI_2}(Y)$, and therefore (1) holds.

\end{proof}

\section{QI$_2$-continuous spaces and strongly QI$_2$-continuous spaces}


In \cite{ASHZ}, Andradi et al. proved that SI-convergence structure in $X$ is topological iff X is I-continuous for any $T_0$ space $X$ having condition (I$^*$).
In \cite{YX} Yang and Xu  introduce the concepts of SI$_2$-convergence and (strongly) I$_2$-continuous space, and proved that SI$_2$-convergence in $X$ is topological iff $X$ is strongly I$_2$-continuous.
 In \cite{RX}, it proved that a poset $P$ is s$_2$-continuous (resp., s$_2$-quasicontinuous) iff the $\mathcal S$-convergence (resp., the $\mathcal{GS}$-convergence) in $P$ is topological.
This naturally raises a question whether there is a characterization of $T_0$ spaces for the GSI$_2$-convergence to be topological. In this section, we shall give such a characterisation.

\begin{definition}\label{I-way below}
Let $X$ be a $T_0$ space, $x\in X$ and $A, B\subseteq X$. We say that $A$ is \emph{I$_2$-way-below} $B$, in symbols $A\ll_{\I_2} B$, if for any irreducible set $D$ in $X$, $ D^{\delta}\cap  B\neq\emptyset$ implies $A\cap\cl D\neq\emptyset$. We write $A\ll_{\I_2} x$ for $A\ll_{\I_2} \{x\}$ and $y\ll_{\I_2} B$ for $\{y\}\ll_{\I_2} B$. For $F\in X^{(<\omega)}$, we write $\dda_{\I_2}F=\{x\in X: x\ll_{\I_2} F\}$ and $\dua_{\I_2}F=\{x\in X: F\ll_{\I_2} x\}$.
\end{definition}

The following fact is simple and the proof is omitted.

\begin{remark}
For a $T_0$ space $X$, Let $G$, $H$, $K$, $M\subseteq X$. Then:

(i) $G\ll_{\I_2} H$ $\Leftrightarrow$ $G\ll_{\I_2} h$ for all $h\in H$;

(ii) $G\ll_{\I_2} H$ $\Leftrightarrow$ $\uparrow G\ll_{\I_2}\uparrow H$;

(iii) $G\ll_{\I_2} H$ $\Rightarrow$ $G\leq H$, that is,$\uparrow H\subseteq \uparrow G$ ;

(iv) $G\leq H\ll_{\I_2} K\leq M$ $\Rightarrow$ $G\ll_{\I_2} M$.
\end{remark}

In the following proposition, we have a connection between I$_2$-way-below relation and GSI$_2$-convergence
structure.

\begin{proposition}\label{I-way below1}
Let $X$ a $T_0$ space. Then $F\ll_{\I_2}x$ if and only if for every net $(x_i)_{i\in I}$, $(x_i)_{i\in I}\stackrel{GSI_2}\longrightarrow x$ implies $x_i\in U$ eventually whenever $F\subseteq U$ for any $U\in \mathcal{O}(X)$.
\end{proposition}

\begin{proof}
Necessity: As $(x_i)_{i\in I}\stackrel{GSI_2}\longrightarrow x$, then there exists $\mathcal{F}\subseteq X^{(<\omega)}$ such that $\{\ua G : G\in\mathcal F\}\in \ir(P_S(X))$ and conditions (i) and (ii) in Definition \ref{GSI-converges} are satisfied. So $\mathop{\bigcap}\limits_{G \in \mathcal{F}} \uparrow G \subseteq\uparrow x $.
Now we show that there is a $G\in \mathcal{F}$ such that $G\subseteq U\in \mathcal{O}(X)$. Suppose not. By Lemma \ref{Rudin irr} and Lemma \ref{Rudin irr open set}, there exists  an
irreducible set $D$ in $X$ such that $D\subseteq \mathop{\bigcup}\limits_{G \in \mathcal{F}} (G\setminus U) $ and  $D\cap (G\setminus U)\neq \emptyset$ for any $G\in \mathcal{F} $.
 Thus, $\mathop{\bigcap}\limits_{d \in D} \uparrow d \subseteq \mathop{\bigcap}\limits_{G \in \mathcal{F}} \uparrow G\subseteq\uparrow x\subseteq U$ and $x\in U\cap D^{\delta}$. By $F\ll_{\I_2}x$, we have that $F \cap\cl D\neq\emptyset$.
By $F\subseteq U\in \mathcal{O}(X)$, we know that $U \cap\cl D\neq\emptyset$. Thus, $U \cap D\neq\emptyset$ which is
a contradiction. Therefore, there is a $G\in \mathcal{F}$ such that $G\subseteq U$. As $(x_i)_{i\in I}\stackrel{GSI_2}\longrightarrow x$, we have that $x_i\in U$ eventually.

Sufficiency: Let $D\in \ir(X)$ and $x\in D^{\delta}$. By Lemma \ref{Irr net}, there is a net $(x_i)_{i\in I}\subseteq D$ such that $(x_i)_{i\in I}\stackrel{GSI_2}\longrightarrow x$.
Let $F\subseteq U$ for any $U\in \mathcal{O}(X)$.
By hypothesis, $x_i\in U$ eventually, thus $D\cap U\neq\emptyset$, hence $F\cap \cl D\neq\emptyset$.  Therefore, $F\ll_{\I_2}x$.
\end{proof}

\begin{definition}\label{QI-continous}
A $T_0$ space $X$ is called $QI_2$-continuous iff for any $x\in X$, the following conditions
are satisfied:

(1) $w(x) = \{\ua F: F \in X^{(<\omega)}$ and $F \ll_{\I_2} x\}$ is an
irreducible set in $P_S(X)$;

(2) $\ua x=\bigcap  w(x)$;
\end{definition}

\begin{proposition}\label{Irr sup}
For a $T_0$ space $X$, the following two conditions are equivalent:

(1) $X$ is $QI_2$-continuous.

(2) For any $x\in X$, there exists $\mathcal{F}\in Irr(P_s(X))$ such that $\mathcal{F}\subseteq w(x)$ and $\uparrow x=\bigcap \mathcal{F}$.
\end{proposition}

\begin{proof}
(1) $\Rightarrow$ (2): Let $\mathcal{F}=w(x)$. Then $\mathcal{F}\in Irr(P_s(X))$ and $\uparrow x=\bigcap  \mathcal{F}$.

(2) $\Rightarrow$ (1): Let $B$ with $\mathcal{F}\subseteq \diamond B$ be a closed set in $X$. By lemma \ref{Rudin irr closed set} and \ref{Rudin irr}, there exists $D\in Irr(X) $ with $D\subseteq \mathop{\bigcup}\limits_{\ua G \in \mathcal{F}}(G\cap B) $ such that  $D\cap (G\cap B)\neq \emptyset$ for any $\ua G\in \mathcal{F} $.
Then $\mathop{\bigcap}\limits_{d \in D} \uparrow d \subseteq \mathop{\bigcap} \mathcal{F} \subseteq \uparrow x$ and $x\in D^{\delta}$.
For any $\ua F\in w(x)$, we have $F\ll_{\I}x$. Thus $\cl D\cap F\neq\emptyset$, and then $B\cap F\neq\emptyset$. Hence $w(x)\subseteq \diamond B$. So $\mathcal{F}\subseteq w(x)\subseteq \cl_{P_s(X)}\mathcal{F}$. As $\mathcal{F}\in \ir(P_S(X))$, then $w(x)\in \ir(P_S(X))$. Since $\mathcal{F}\subseteq w(x)$ and $\uparrow x=\bigcap  \mathcal{F}$, $\uparrow x\subseteq \bigcap  w(x)\subseteq \bigcap  \mathcal{F}=\uparrow x$, and then $\bigcap  w(x)$. So $X$ is $QI_2$-continuous.
\end{proof}

\begin{proposition}\label{Irr sup2}
Let $X$ be a $T_0$ space, $x\in X$, $F\in X^{(<\omega)}$ and $(x_i)_{i\in I}$ be a net in $X$. Consider the following two conditions:
\begin{enumerate}[\rm (1)]
\item  $(x_i)_{i\in I}\stackrel{GSI_2}\longrightarrow x$.

\item  For any $F\ll_{\I_2} x$, we have $x_i\in U$ eventually whenever $F\subseteq U$ for any $U\in \mathcal{O}(X)$.
\end{enumerate}
\noindent Then \emph{(1)} $\Rightarrow$ \emph{(2)}, and two conditions are equivalent if $X$ is $QI_2$-continuous.
\end{proposition}

\begin{proof}

(1) $\Rightarrow$ (2): By Proposition \ref{I-way below1}.

(2) $\Rightarrow$ (1): Suppose that $X$ is $QI_2$ continuous. Then $w(x)\in Irr(P_s(X))$ and $\bigcap w(x)=\uparrow x$. For any   $U\in\mathcal O(X)$, if $w(x)\cap \Box U\neq\emptyset$, then there is $\ua F\in w(x)$ with $F\subseteq U$ such that $F\ll_{I_2} x$, and hence $x_i\in U$ eventually by (2). Thus $(x_i)_{i\in I}\stackrel{GSI_2}\longrightarrow x$.

\end{proof}

\begin{proposition}\label{topo-QIC}
Let $X$ be a $T_0$ space. If $GSI_2$-convergence in $X$ is topological, then $X$ is $QI_2$-continuous.
\end{proposition}

\begin{proof}
By Lemma \ref{SItopo=GSItopo}, $\mathcal O(\mathcal{GSI}_2(X))=\mathcal O_{\SI_2}(X)$. Thus if $GSI_2$-convergence in $X$ is topological, we must have $(x_i)_{i\in I}\stackrel{GSI_2}\longrightarrow x$ iff $(x_i)_{i\in I}$ converges to $x$ with respect to the topology $\mathcal O_{\SI_2}(X)$.
Let $x\in X$. Define $I=\{(U, e)\in\mathcal N_{\SI_2}(x)\times X : e\in U\}$, where $\mathcal N_{\SI_2}(x)$ consists of all open sets containing $x$ in the space $(X, \mathcal O_{\SI_2}(X))$,  and equip $I$ with $\leq$ defined as follows: $(U_1, e_1)\leq (U_2, e_2)$ iff $U_1\supseteq U_2$.
Thus $I$ is a directed set. Define $x_{(U,e)}=e$ for any $(U,e)\in I$. It is easy to see that the net $(x_i)_{i\in I}$ converges to $x$ in $(X, \mathcal O_{\SI_2}(X))$ and hence $(x_i)_{i\in I}\stackrel{GSI_2}{\longrightarrow} x$. So there exists an irreducible set $\mathcal{F}\in \ir(P_S(X))$ such that $(x_i)_{i\in I}$ and $\mathcal{F}$ satisfy conditions (i) and (ii) of Definition \ref{GSI-converges}.

Now we prove that $\mathcal{F}\subseteq w(x)$.
Suppose that $\ua F\in \mathcal{F}$. We verify that $F\ll_{\I_2} x$.
Let $D\in \ir(X)$ and $x\in D^{\delta}$. In order to prove $F\cap\cl D\neq\emptyset$, we only need to prove $U\cap D\neq\emptyset$ for any $F\subseteq U\in \mathcal{O}(X)$. 
As  $(x_{(U, e)})_{(U, e)\in I}\stackrel{GSI_2}{\longrightarrow} x$  and $F\subseteq U\in \mathcal{O}(X)$, by (i) of Definition \ref{GSI-converges} we can know that $x_i\in U$ eventually, and hence there is a SI$_2$-open set $V$ such that $x\in V \subseteq U$. By Lemma \ref{Irr net}, there exists a net $(y_j)_{j\in J}\subseteq D$
such that it GSI$_2$-convergence to $x$ and hence it converges to $x$ in the space $(X, \mathcal O_{\SI_2}(X))$ by Remark \ref{GSIremark}. Then
$y_j\in V$ eventually, and consequently, $y_j\in U$ eventually. Hence $U\cap D\neq\emptyset$. So $\mathcal{F}\subseteq w(x)$. Let $\ua F \in \mathcal{F}$. we know $x\in \uparrow F$, and then $\uparrow x\subseteq\bigcap\mathcal{F}$.
As $\uparrow x\subseteq \bigcap\mathcal{F}\subseteq\uparrow x$, we have $\bigcap\mathcal{F}=\uparrow x$. By Proposition \ref{Irr sup}, $X$ is $QI_2$-continuous.

\end{proof}

\begin{definition}\label{SQICont}
A $T_0$ space $X$ is called \emph{strongly QI$_2$-continuous} if the following  conditions hold:

(i) $X$ is QI$_2$-continuous, and

(ii) For any $F\in X^{(<\omega)}$, $x\in X$ with $F\ll_{\I_2}x$ and $U\in\mathcal O(X)$ with $F\subseteq U$, there exists an SI$_2$-open set $W$ with $x\subseteq W\subseteq U$.
\end{definition}

\begin{remark}
For a dcpo $P$, considering the topological space $(P, \alpha (P))$ , we have that $P$ is $s_2$-quasicontinuous iff $(P, \alpha (P))$ is strongly QI$_2$-continuous.
\end{remark}

\begin{proposition}\label{topo-SQIC}
Let $X$ be  a $T_0$ space. If $GSI_2$-convergence in $X$ is topological, then $X$ is strongly $QI_2$-continuous.
\end{proposition}

\begin{proof} Firstly, by Proposition \ref{topo-QIC} $X$ is $QI_2$-continuous. Suppose that $F\ll_{\I_2}x$, $U\in\mathcal O(X)$ and $F\subseteq U$. Then $x\in\dua_{\I_2}F\subseteq U$. Consider the net $(x_i)_{i\in I}$ similarly defined in the proof of Proposition \ref{topo-QIC}, where $I=\{(U, e)\in\mathcal N_{\SI_2}(x)\times X : e\in U\}$.  Then $(x_i)_{i\in I}$ converges to $x$ in $(X, \mathcal O_{\SI_2}(X))$ and  $(x_i)_{i\in I}\stackrel{GSI_2}{\longrightarrow} x$. By Proposition \ref{Irr sup2}, $x_i\in U$ eventually and thus exists a SI$_2$-open set $W$ with $x\subseteq W\subseteq U$. Therefore, $X$ is strongly $QI_2$-continuous.
\end{proof}

\begin{proposition}\label{SQIC-topo}
If $X$ is a strongly $QI_2$-continuous space, then $GSI_2$-convergence in $X$ is topological.
\end{proposition}

\begin{proof}
Let $(x_i)_{i\in I}$ be a net in $X$ and $x\in X$. If $(x_i)_{i\in I}\stackrel{GSI_2}\longrightarrow x$ and $x\in U\in \mathcal O_{\SI_2}(X)$, there exists an irreducible set $\mathcal{F}=\{\ua F:F\in X^{(<\omega)}\}\in \ir(P_S(X))$ such that $\bigcap \mathcal{F}\subseteq \uparrow x\subseteq U$. By Proposition \ref{rudin yunyong}, there exists an set $\ua F\in \mathcal{F}$ such that $\ua F\subseteq U$, that is, $\mathcal{F}\cap \Box U\neq\emptyset$. Thus, $x_i\in U$ eventually and $(x_i)_{i\in I}$ converges to $x$ in $(X, \mathcal O(\mathcal{GSI}_2(X)))=(X, \mathcal O_{\SI_2}(X))$.

Conversely, suppose that $(x_i)_{i\in I}$ converges to $x$ in $(X, \mathcal O(\mathcal{GSI}_2(X)))$. Then by Lemma \ref{SItopo=GSItopo}, $(x_i)_{i\in I}$ converges to $x$ with respect to the topology $\mathcal O_{\SI_2}(X)$. Now we prove that $(x_i)_{i\in I}\stackrel{GSI_2}\longrightarrow x$. Since $X$ is a strongly $QI_2$-continuous space, $w(x)\in Irr(P_s(X))$ and  $\uparrow x=\bigcap w(x)$. Let $V\in \mathcal{O}(X)$ with $w(x)\cap \Box V\neq\emptyset$.  Hence there is a set $\ua F\in w(x)$ such that $\ua F\subseteq V$. As $X$ is a strongly $QI_2$-continuous space, there is a SI$_2$-open $W$ such that $x\in W\subseteq V$, and hence $x_i\in W$ eventually. Therefore, $x_i\in V$ eventually, that is, $(x_i)_{i\in I}\stackrel{GSI_2}\longrightarrow x$. Thus $GSI_2$-convergence is topological.

\end{proof}

By  Proposition \ref{SQIC-topo} and Proposition \ref{topo-SQIC}, we get the main result of this paper.

\begin{theorem}\label{3tl23}
For an  $T_0$ space $X$, the following conditions are equivalent:
\begin{enumerate}[\rm (1)]
\item  The $GSI_2$-convergence structure in $X$ is topological.

\item  For any net $(x_i)_{i\in I}$ in $X$ and $x\in X$, $(x_i)_{i\in I}\stackrel{GSI_2}\longrightarrow x$ iff $(x_i)_{i\in I}$ converges to $x$ with respect to the SI$_2$-topology $\mathcal O_{\SI_2}(X)$.

\item  $X$ is strongly $QI_2$-continuous.
 \end{enumerate}
\end{theorem}

For a poset $P$, considering the topological space $(P, \alpha (P))$ , we have the following fact:

\begin{corollary} \emph{(\cite{RX})}
 For a poset $P$, the following conditions are equivalent:
\begin{enumerate}[\rm (1)]
\item  $GS_2$-convergence structure in $P$ is topological.

\item  For any net $(x_i)_{i\in I}$ in $X$ and $x\in X$, $(x_i)_{i\in I}$ $GS_2$--converges to $x$ iff $(x_i)_{i\in I}$ converges to $x$ with respect to the weakly Scott topology $\sigma_2(P)$.

\item  $P$ is $s_2$-quasicontinuous.
\end{enumerate}
\end{corollary}

\end{document}